\documentclass{amsart}

\newtheorem{thm}{Theorem}[section]
\newtheorem{cor}[thm]{Corollary}
\newtheorem{lem}[thm]{Lemma}
\newtheorem{prop}[thm]{Proposition}

\theoremstyle{definition}

\newtheorem{rem}{Remark}[section]

\newtheorem*{ack}{Acknowledgements}

\theoremstyle{remark}

\let\e\varepsilon

 \newcommand{\U}{{\mathcal{U}}}
\newcommand{\bB}{{\mathbb{B}}}
\newcommand{\bM}{{\mathbb{M}}}

\newcommand{\Tr}{\operatorname{Tr}}

\newcommand{\id}{{\operatorname{id}}}

\begin{document}
\title[]{Homomorphisms with small bound between Fourier algebras
}

\author[Y. Kuznetsova]{Yulia Kuznetsova}
\address{Laboratoire de Math\'ematiques de Besan\c con\\
Universit\'e Bourgogne Franche-Comt\'e \\
16 route de Gray\\
 25030 Besan\c con Cedex, France}
\email{yulia.kuznetsova@univ-fcomte.fr 
    }

\author[J. Roydor]{Jean Roydor 
}
\address{Institut de Math\'ematiques de Bordeaux  \\
Universit\'e de Bordeaux  \\
 351 Cours de la Lib\'eration \\
 33405 Talence Cedex, France}
\email{jean.roydor@math.u-bordeaux1.fr 
}

\begin{abstract} Inspired by Kalton\&Wood's work on group algebras, we describe almost completely contractive algebra homomorphisms from Fourier algebras into Fourier-Stieltjes algebras (endowed with their canonical operator space structure). We also prove that two locally compact groups are isomorphic if and only if there exists an algebra isomorphism $T$ between the associated Fourier algebras (resp.  Fourier-Stieltjes algebras)
 with completely bounded norm $\| T \|_{cb} < \sqrt {3/2}$ (resp. $ \| T \|_{cb} < \sqrt {5}/2$).  We show similar results involving the norm distortion $\| T \| \| T ^{-1} \|$ with universal but non-explicit bound. Our results subsume Walter's well-known structural theorems and also Lau's theorem on second conjugate of Fourier algebras.

\end{abstract}
\maketitle



\section{Introduction and notation}

\indent To a locally compact group, one can associate several different algebras. The general question is: does this algebra remember the group? Or more precisely, which relation does one need between two algebras to be able to identify the groups? For instance, let $G$ and $H$ be two locally compact groups; it is a classical result of J. Wendel \cite{We2} that $G$ and $H$ are isomorphic as topological groups if and only if there exists a contractive algebra isomorphism between the group algebras $L_1(G)$ and $L_1(H)$ (equipped with the convolution product). This improves his earlier paper \cite{We} concerning isometries, note also that B.E. Johnson \cite{J2} and R. Rigelhof \cite{Rig} proved analogous results for measure algebras. It is classical in Banach space theory to look for stability of isometric results (see e.g. \cite{A}, \cite{Ca} and \cite{B}). In this vein, N. Kalton and G. Wood  improved Wendel's result in weakening the relation between the group algebras. The algebraic hypothesis in Wendel's result can not be dropped, because surjective linear isometries between $L_1$-spaces only determine the underlying measure spaces. But one can relax the restriction on the norm of the algebra isomorphism: in \cite{KW}, N. Kalton and G. Wood showed that $G$ and $H$ are isomorphic if and only if there exists an algebra isomorphism $T$ of
$L_1(G)$ onto $L_1(H)$ with $\| T \|<  \gamma\approx 1.246$. In the case of two locally compact  abelian groups, they show that the bound can be improved to equal $\sqrt{2}$ and is actually optimal, see Example 1 in \cite{KW}.\\
 \indent In this paper, we are interested in the Fourier algebra $A(G)$ and the Fourier-Stieltjes algebra $B(G)$ associated to a locally compact group $G$ (defined by P. Eymard in \cite{E}). These algebras are isometric respectively to the predual of the group von Neumann algebra $VN(G)$ and to the predual of the universal von Neumann algebra $W^*(G)$ of the group (in particular, they are noncommutative $L_1$-spaces), see the end of the section for more details. We just mention here that in the case of an abelian group $G$, the algebra $A(G)$, respectively $B(G)$, is identified (via the Fourier transform) with the group algebra $L_1(\hat G)$, respectively with the measure algebra $M(\hat G)$ (where $\hat G$ denotes the dual group of $G$).\\
 \indent It is a well-known result of M. Walter \cite{Wa} that the Fourier algebras $A(G)$ and $A(H)$ are isometrically isomorphic as Banach algebras if and only if $G$ and $H$ are isomorphic as topological groups. Actually contractivity of the algebra isomorphism is sufficient (see Corollary 5.4 \cite{Ph}),  which is analogous to Wendel's results. M. Walter proved a similar theorem for Fourier-Stieltjes algebras, which is thus the analog of Johnson's result mentioned above. Inspired by the improvement of N. Kalton and G. Wood described above, one can wonder whether the isometric or contractive assumption on the algebra isomorphism in Walter's result is really needed to recover the groups structure? To our knowledge, this question has never been studied before. All the known results (see \cite{Wa}, \cite{Ph} and \cite{IS}) on homomorphisms of Fourier algebras require posivity, contractivity or complete contractivity of the homomorphism, except Theorem 3.7 \cite{IS} which needs in return  the amenability of the group and complete boundedness of the algebra homomorphism.\\
 \indent In Section 2, we consider Fourier algebras and Fourier-Stieltjes algebras as completely contractive algebras, i.e. endowed with their canonical operator space structure, E. Effros and Z.-J. Ruan were the first to consider these operator space structures (see \cite{ER1} for more details). This means in particular that instead of the class of all bounded operators, we consider a smaller subclass of completely bounded operators with the completely bounded norm $\|\cdot\|_{cb}$ (see \cite{BLM}, \cite{P} or \cite{Pa} for general theory of operator spaces). Usually, analogues between group algebras and Fourier algebras work better in the category of operator spaces (see e.g. the nice result of \cite{Ru}). Therefore, we first treat the almost completely contractive case (note that $\| T \| \le \| \id_{\mathbb{M}_2} \otimes T \| \le \|T\|_{cb}$):

\begin{thm} \label{main1} Let $G$ and $H$ be locally compact groups.
\begin{enumerate}
\item Let $T:A(G) \to B(H)$ be a nonzero algebra homomorphism. Suppose that  $\| \id_{\mathbb{M}_2} \otimes T \|  <\sqrt {5}/{2}$, then there exist an open subgroup $\Omega$ of $H$, $t_0 \in G$, $h_0 \in H$ and a continuous group morphism $\tau:\Omega \to G$ such that
$$T(f)(h)=  \begin{cases}
 f(t_0  \tau(h_0 h)) & \text{if\, } h \in h_0^{-1}\Omega\\
0 & \text{if\, not},
\end{cases}  $$
for any $f  \in A(G)$, $h \in H$. Hence
$T$ is actually completely contractive, i.e.~$\|T\|_{cb}\le 1$.
\item Let $T:A(G) \to A(H)$ be a surjective algebra isomorphism between the Fourier algebras. If  $\| \id_{\mathbb{M}_2} \otimes T \|  < \sqrt {{3}/{2}}$, then there exist $t_0 \in G$ and a topological isomorphism $\tau:H \to G$ such that $$T(f)(h)=f(t_0  \tau(h)),$$ for any $f  \in A(G)$, $h \in H$. Hence
$T$ is actually completely isometric.
\item Let $T:B(G) \to B(H)$ be a surjective algebra isomorphism between Fourier-Stieltjes algebras. If  $\| \id_{\mathbb{M}_2} \otimes T \| <\sqrt {5}/{2}$, the same conclusion as in 2.~holds (for any $f \in B(G)$).
  \end{enumerate}
  \end{thm}

The important point here is that our result enables us to observe the new curious phenomenon of a ``norm gap'': for any nonzero algebra homomorphism $T:A(G) \to B(H)$,  we have either $\|  T \|_{cb}  =1$ or $\|  T \|_{cb}  \geq \sqrt 5/{2}$. Idem for cases (2) and (3). This can be compared with the result of N. Kalton and G. Wood \cite{KW}: for $G,H$ abelian groups, if $T:L_1(G)\to L_1(H)$ is a surjective algebra isomorphism then either $T$ is isometric or $\|T\|\ge (1+\sqrt3)/2$.\\
 \indent In Section 3, we consider Fourier algebras and Fourier-Stieltjes algebras as Banach algebras (and the usual operator norm of $T$). We are still able to prove that Walter's result is stable but we need a hypothesis on the norm distortion now:

\begin{thm}\label{main2} There exists a universal constant $\varepsilon_0>0$ such that for any locally compact groups $G$ and $H$,
\begin{enumerate}
\item if $T:A(G) \to A(H)$ is a surjective algebra isomorphism between the Fourier algebras and $\|  T \| \|  T^{-1} \| <1+\varepsilon_0$, then there exist $t_0 \in G$ and a topological isomorphism or anti-isomorphism $\tau:H \to G$ such that for any $f \in A(G)$, $h \in H$,
$$T(f)(h)=f(t_0  \tau(h)),$$ hence $T$ is actually isometric.
\item The same result holds for surjective algebra isomorphisms between Fourier-Stieltjes algebras.
\end{enumerate}
\end{thm}
This theorem has an interesting application. In \cite{L}, A.T.-M. Lau considers the second conjugate $A(G)^{**}$ equipped with the Arens product. He proved (see Theorem 5.3  \cite{L}) that two discrete groups $G$ and $H$ are isomorphic if there is an order preserving isometric algebra isomorphism $T:A(G)^{**} \to A(H)^{**}$. Our Corollary \ref{lau} improves his theorem, even in the isometric case, as we do not need any assumption on the order structure.\\
 \indent Our proofs are by duality, they are entirely different from Kalton and Wood's ones (as we want to deal with non-abelian groups). For the proof of our Theorem \ref{main1}, we use two $2 \times 2$ matrix tricks and for Theorem \ref{main2} we use ultraproducts of $C^*$-algebras, what allows us to obtain a universal but non-explicit bound. It is important to notice that all the known structural results (see \cite{Wa}, \cite{Ph}) rely on the linear independence of the $\lambda_g$'s in $VN(G)$ (see notation below). Our novelty is to deal with homomorphisms with norm greater than one, we need more precision and we base our proofs on the fact the $\lambda_g$'s form a uniformly discrete subset of $VN(G)$.\\
\indent We finish this section with some notation. The reader is referred to \cite{E} for details on Fourier and Fourier-Stieltjes algebras and to \cite{ER1} for their operator space structure. If $K$ is a Hilbert space, then $\mathbb{B}(K)$ denotes the algebra of all bounded operators on $K$ equipped with the operator norm and $\mathbb{U}(K)$ the group of unitary operators on it. We denote $\omega:G \to \mathbb{U}({K}_\omega)$ the universal representation of $G$ and $W^*(G)$ is the von Neumann subalgebra of $\mathbb{B}(K_\omega)$ generated by $\omega (G)$.
We recall that $B(G)$ is the set of all functions defined on $G$ of the form $g \mapsto \langle  \omega_g ( \xi), \eta \rangle$, where $\xi, \eta \in {K}_\omega$. Then $B(G)$ is a completely contractive algebra (in the sense of \cite{ER1}) for the pointwise multiplication and the operator structure herited from the duality $B(G)^*=W^*(G)$. We denote $\lambda:G \to \mathbb{U}(L_2(G))$ the left regular representation and $VN(G)$ the von Neumann algebra generated by $\lambda (G)$ inside $\mathbb{B}(L_2(G))$.
We recall that $A(G)$ is the ideal of $B(G)$ consisting  of all functions defined on $G$ of the form $g \mapsto \langle  \lambda_g ( \xi), \eta \rangle$, where $\xi, \eta \in L_2(G)$. Then $A(G)$ is also a completely contractive algebra and $A(G)^*=VN(G)$. For $f \in B(G)$, $g \in G$, the left translation of $f$ by $g$ is denoted $g \cdot f$, i.e. $(g \cdot f) (t)=f (g^{-1}t)$, $t \in G$. Finally, for a Banach algebra $A$, we denote $\sigma(A)$ the spectrum of $A$ i.e. the set of all nonzero multiplicative linear functionals on $A$.

\begin{ack} This work started while the second author was visiting the University of Bourgogne--Franche-Comt\'e at Besan\c con during Fall 2014; he would like to thank its Functional Analysis team for the wonderful organization of the thematic program ``Geometric and noncommutative methods in functional analysis". He is also grateful to the Functional Analysis team at IMJ-PRG (Universit\'e Pierre et Marie Curie, Paris), especially Gilles Godefroy and Yves Raynaud, for hosting him since Fall 2015.
\end{ack}

\section{The almost completely contractive case}

In the next two lemmas, $K$ denotes a Hilbert space. As usual in operator space theory, we identify isometrically $\bM_2(\bB(K))=\bB(K \oplus^2 K)$ (where $\oplus^2$ is the Hilbert space direct sum).
\begin{lem}\label{invmult} Let $u$ be a unitary in $\bB(K)$.
 Let $x \in \bB(K)$ and $c \geq 1 $ such that
$$\left\| \left[ \begin{array}{cc} u  & 1  \\ -1 & x \end{array}\right] \right\| \leq c\sqrt{2},$$
then $\big\| x-u^* \big\| \leq 2\sqrt{c^2-1} .$
\end{lem}
\begin{proof}
As $$ \left[ \begin{array}{cc} u^* & 0 \\ 0 &
    1 \end{array}\right] \left[ \begin{array}{cc} u & 1 \\ -1 &
    x \end{array}\right] \left[ \begin{array}{cc} 1 & 0 \\ 0 &
    u \end{array}\right]= \left[ \begin{array}{cc} 1 & 1 \\ -1
    & xu \end{array}\right],$$ we have
    $$ \left\|
\left[ \begin{array}{cc} u & 1 \\ -1 &
    x \end{array}\right] \right\| = \left\|  \left[ \begin{array}{cc} 1 & 1 \\ -1
    & xu \end{array}\right] \right\|.$$
    Let $k\in K$, then $$ \left\|
\left[ \begin{array}{cc} 1 & 1 \\ -1
    & xu \end{array}\right]
\left[ \begin{array}{c} k \\ k \end{array}\right]\right\| \leq
c\sqrt{2} \left\| \left[ \begin{array}{c} k
    \\ k \end{array}\right]\right\|.$$ Hence $\| xu(k)-k \|^2
+ 4\| k \|^2 \leq 4c^2\|k\|^2$, which gives $\| xu-1 \| \leq
2\sqrt{c^2-1}$.
\end{proof}

\begin{lem}\label{unitmult} Let $u,\, v$ be two unitaries in $\bB(K)$.
 Let $x \in \bB(K)$ and $c \geq 1 $ such that
$$\left\| \left[ \begin{array}{cc} u  & x  \\ -1 & v \end{array}\right] \right\| \leq c\sqrt{2},$$
then $\big\| x-uv \big\| \leq 2\sqrt{c^2-1} .$
\end{lem}
\begin{proof}
Note that $$ \left[ \begin{array}{cc} u^* & 0 \\ 0 &
    1 \end{array}\right] \left[ \begin{array}{cc} u & x \\ -1 &
    v \end{array}\right] \left[ \begin{array}{cc} 1 & 0 \\ 0 &
    v^* \end{array}\right]= \left[ \begin{array}{cc} 1 & u^*xv^* \\ -1
    & 1 \end{array}\right],$$ hence without loss of generality we can
assume that $u=v=1$. Take $k\in K$, then $$ \left\|
\left[ \begin{array}{cc} 1 & x \\ -1 & 1 \end{array}\right]
\left[ \begin{array}{c} -k \\ k \end{array}\right]\right\| \leq
c\sqrt{2} \left\| \left[ \begin{array}{c} -k
    \\ k \end{array}\right]\right\|.$$ As in the previous lemma, this implies $\| x-1 \| \leq
2\sqrt{c^2-1}$.
\end{proof}

We are ready to prove Theorem \ref{main1}.\\
\begin{proof}[Proof of 1.\ of Theorem \ref{main1}]
The first step of the proof is to replace $T$ by a suitable unital algebra homomorphism. Let $h \in H$ be viewed (via the universal representation) as a character of $B(H)$, i.e. $\omega_h \in \sigma(B(H)) \subset W^*(H)$. As $T$ is an algebra homomorphism, $T^*(\omega_h) \in \sigma(A(G)) \cup \{0\}$. It is well-known that  the spectrum of a Fourier algebra can be identified with the underlying group, i.e. $\sigma(A(G)) \simeq G$ (via the regular representation, see \cite{E}). Hence we consider the subset of $H$ defined by
$$\Omega_T=\{ h \in H ~:~  T^*(\omega_h) \in G \}=\{ h \in H ~:~  T^*(\omega_h) \neq 0 \}.$$
As $T$ is nonzero, the set $\Omega_T$ is not empty. Let $h_0 \in \Omega_T$ and define an algebra homomorphism $S:A(G) \to B(H)$ by $$S(f)=h_0 \cdot T(T^*(h_0) \cdot f), \qquad f \in A(G).$$
Then $S^*: W^*(H) \to VN(G)$ is a linear isomorphism which is unital,
 \begin{align*}S^*(\omega_{e_H})(f) = (Sf)(e_H) =&  h_0^{-1}\! \cdot \,T\big(T^*(h_0)  \cdot f \big)(e_H)\\
=& T\big(T^*(h_0)  \cdot f \big)(h_0) \\
 = & (T^*(h_0)  \cdot f)(T^*(h_0) )= f(e_G). \end{align*}
Since the left translation are completely isometric on Fourier-Stieltjes algebras, we have
$$\| \id_{\mathbb{M}_2} \otimes S^* \|=\| \id_{\mathbb{M}_2} \otimes T^* \| =\| \id_{\mathbb{M}_2} \otimes T \| <\frac {\sqrt 5}{2}.$$
As it was done for $T$, consider
$$\Omega_S=\{ h \in H ~:~  S^*(\omega_h) \in G \}=\{ h \in H ~:~  S^*(\omega_h) \neq 0 \}.$$
The last equality shows that $\Omega_S$ is open. But actually, as $S$ is obtained from $T$ by composition with two left translations and since the adjoint of a left translation preserves the group, we actually have $\Omega_S=\Omega_T$.
Denote $s:\Omega_S \to G$ the restriction of $S^*$ to  $\Omega_S$, i.e. for $h \in \Omega_S$, $S^*(\omega_{h})=\lambda_{s(h)}$ . We want to show that $\Omega_S$ is a subgroup of $H$ and $s$ is a continuous group homomorphism. The map $s$ is clearly continuous, as the restriction of a $w^*$-continuous map. Let us first prove that $\Omega_S$ is stable under taking inverses. Let $h \in \Omega_S$,
 \begin{align*} \left\| \left[ \begin{array}{cc} \lambda_{s(h)} & 1 \\ -1 &
    S^*(\omega_{h^{-1}}) \end{array}\right]  \right\| &=
    \left\| \left[ \begin{array}{cc} S^*(\omega_{h}) & S^*(1) \\ S^*(-1) &
    S^*(\omega_{h^{-1}}) \end{array}\right]  \right\| \\
    &\leq \| \id_{\mathbb{M}_2} \otimes S^* \|
    \left\| \left[ \begin{array}{cc} \omega_{h} & 1 \\ -1 &
    \omega_{h^{-1}} \end{array}\right]  \right\|. \end{align*}
    But $$\left\| \left[ \begin{array}{cc} \omega_{h} & 1 \\ -1 &
    \omega_{h^{-1}} \end{array}\right]  \right\|=\sqrt{2},$$ so we can apply Lemma \ref{invmult} to obtain the following inequality
    $$\big\| S^*(\omega_{h^{-1}})-\lambda_{s(h)^{-1}} \big\| \leq 2\sqrt{\| \id_{\mathbb{M}_2} \otimes S^* \|^2-1}<1.$$
    As $\lambda_{s(h)^{-1}}$ is on the unit sphere of $VN(G)$, $S^*(\omega_{h^{-1}})$ is not zero, that is $h^{-1} \in \Omega_S$ and $S^*(\omega_{h^{-1}})=\lambda_{s(h^{-1})}$. Hence
    $$\big\| \lambda_{s(h)^{-1}}-\lambda_{s(h^{-1})} \big\| <1.$$ If $\lambda_{s(h)^{-1}}$ and $\lambda_{s(h^{-1})}$ were distinct, then we would have $\| \lambda_{s(h)^{-1}}-\lambda_{s(h^{-1})} \|_{VN(G)} \geq \sqrt 2 $ (see Lemma \ref{norm-in-vng}), so  $\lambda_{s(h)^{-1}} = \lambda_{s(h^{-1})}$.\\
Now let us prove that $\Omega_S$ is stable by multiplication and $s$ is multiplicative. Let $h_1, h_2 \in \Omega_S$,
\begin{align*} \left\| \left[ \begin{array}{cc} \lambda_{s(h_1)} & S^*(\omega_{h_1h_2}) \\ -1 &
    \lambda_{s(h_2)} \end{array}\right]  \right\| &=
    \left\| \left[ \begin{array}{cc} S^*(\omega_{h_1}) & S^*(\omega_{h_1h_2}) \\ S^*(-1) &
    S^*(\omega_{h_2}) \end{array}\right]  \right\| \\
    &\leq
     \| \id_{\mathbb{M}_2} \otimes S^* \|
    \left\| \left[ \begin{array}{cc} \omega_{h_1} & \omega_{h_1h_2} \\ -1 &
    \omega_{h_2} \end{array}\right]  \right\| \end{align*}
    and $$\left\| \left[ \begin{array}{cc} \omega_{h_1} & \omega_{h_1h_2} \\ -1 &
    \omega_{h_2} \end{array}\right]  \right\|=\sqrt{2},$$ we can apply Lemma \ref{unitmult} to obtain
    \begin{align*} \big\| S^*(\omega_{h_1h_2})-\lambda_{s(h_1)s(h_2)} \big\|  \leq 2\sqrt{\| \id_{\mathbb{M}_2} \otimes S^* \|^2-1}<1.\end{align*}
    As above, we get $S^*( \omega_{h_1h_2})$ is not zero, so the product $h_1h_2 \in \Omega_S$ and 
     \begin{align*} \big\| \lambda_{s(h_1 h_2)}-\lambda_{s(h_1)s(h_2)} \big\|&=\big\| \lambda_{s(h_1 h_2)}-\lambda_{s(h_1)}\lambda_{s(h_2)} \big\|\\ &=\big\| S^*( \omega_{h_1h_2})-\lambda_{s(h_1)s(h_2)} \big\|  <1. \end{align*}
Therefore $s(h_1h_2)=s(h_1)s(h_2)$ by Lemma \ref{norm-in-vng}. Finally, $\Omega_S $ is a subgroup of $H$ and $s$ is a group homomorphism, this finishes the proof with $\Omega=\Omega_S $, $t_0 =T^*(h_0)$ and $\tau=s$. By Proposition 3.1 in \cite{IS}, $T$ is completely contractive.
\end{proof}
\begin{proof}[Proof of 2.\ of Theorem \ref{main1}]
Denote by $t$ the restriction of $T^*$ to the spectrum of $A(H)$ (which can be identified with $H$). As $T$ is an algebra isomorphism, $t$ is a homeomorphism between $H$ and $G$ (not necessarily group morphism). Let us define $S:A(G) \to A(H)$ by $$S(f)=T(t(e_H) \cdot f), \qquad f \in A(G).$$
Then $S^*: VN(H) \to VN(G)$ is a unital linear isomorphism and it satisfies
$\| \id_{\mathbb{M}_2} \otimes S^* \| <\sqrt {{3}/{2}}$. The restriction of $S^*$ to the spectrum of $A(H)$, denoted by $s:H \to G$ is also a homeomorphism. Form here we just need to prove that $s$ is a group morphism. Replacing the universal representation by the regular one in the second part of the proof of 1. Theorem \ref{main1}, the computation lead to $$\big\| \lambda_{s(h_1h_2)}-\lambda_{s(h_1)s(h_2)} \big\| \leq 2\sqrt{\| \id_{\mathbb{M}_2} \otimes S^* \|^2-1}<\sqrt 2.$$
By Lemma  \ref{norm-in-vng}, this gives $s(h_1h_2)=s(h_1)s(h_2)$ and it finishes the proof with $t_0=t(e_H)$ and $\tau=s$.
\end{proof}

\begin{proof}[Proof of 3.\ of Theorem \ref{main1}] As in case 2., set similarly $t: \sigma(B(H)) \to \sigma(B(G))$, $t(v)=T^*(v)$. A priori, $t$ might not map $H$ into $G$. Now for $h \in H$, $T^*(h) \in \sigma(B(G))$ and (by \cite{Wa} Theorem 1) there are two cases: either $T^*(h) \in G$ or $T^*(h) \in A(G)^0$, where $A(G)^0$ denotes the annihilator of $A(G)$ inside $B(G)^*=W^*(G)$. Set $$\Omega_T = \{ h\in H: t(h)\in G\}=\{ h \in H ~:~  t(h) \notin A(G)^0 \}.$$ We claim that $t(\Omega_T)$ is dense in $G$. Otherwise, as $A(G)$ is regular (see \cite{E}), there would exist $f\in A(G) \backslash \{0\}$ such that $f_{| t(\Omega_T)} =0$. But this would imply $T(f)=0$, which is impossible since $T$ is injective.
Hence, $t(\Omega_T)$ is dense in $G$. In particular, this implies that $\Omega_T \not= \emptyset$. Pick $h_0\in \Omega_T$ and as in case 1. define $S:B(G) \to B(H)$ by
$$
S(f) = h_0^{-1}\! \cdot \,T\big(t(h_0)  \cdot  f\big), \qquad f \in B(G).
$$
Then $S^*: W^*(H) \to W^*(G)$ is a unital completely bounded isomorphism.
  Consider again,
$$\Omega_S=\{ h \in H ~:~  S^*(h) \in G \}=\{ h \in H ~:~  S^*(h) \notin A(G)^0 \}$$ and denote $s:\Omega_S \to G$ the restriction of $S^*$ to  $\Omega_S$.
As in 1.,  actually $\Omega_S=\Omega_T$ and the same argument shows that $s(\Omega_S)$ is dense in $G$.
Replacing the regular representation by the universal  one in the proof of 2. Theorem \ref{main1} shows that $\Omega_S$ is a subgroup of $H$ and $s$ is a bicontinuous group isomorphism. Hence, to finish the proof, we just need to show that $\Omega_S = H$. Here we can follow the argument at the end of the proof of Theorem 6.1 in \cite{Ph} (we reproduce it here for completeness), $s(\Omega_S)$ is an open dense locally compact subgroup of $G$, so $s(\Omega_S)=G$. Let us denote $X$ the $w^*$-closure of the span of $s(\Omega_S)$ inside $W^*(H)$. As $s(\Omega_S)$ is dense in $G$, we get $S^*(X)=W^*(G)$, so $X=W^*(H)$. Therefore $\Omega_S = H$, otherwise we could find $f\in A(G) \backslash \{0\}$ such that $f_{| \Omega_S} =0$ (since $\Omega_S$ is open). But this would imply $\langle f,x \rangle=0$ for any $x\in X$, so $f=0$, this is a contradiction.
\end{proof}

\begin{cor} Let $G,H$ be two locally compact groups. Let $T:A(G) \to B(H)$ be an injective algebra homomorphism. If  $\| \id_{\mathbb{M}_2} \otimes T \|  <\sqrt {5}/{2}$, then $T$ is actually completely isometric.
\end{cor}
\begin{proof} As in case 1., we can define a new algebra homomorphism $S$ which is still injective. Following now case 3., injectivity implies that $s(\Omega_S)$ is dense in $G$ and now denoting $M$ the von Neumann algebra generated by $\omega(\Omega_S)$ inside $W^*(H)$, we get that the restriction $S^*_{|M}:M \to VN(G)$ is a surjective $*$-homomorphism. Hence for any $n \in \mathbb{N}$, $\id_{\mathbb{M}_n} \otimes S^*$ maps the unit ball of $\mathbb{M}_n(W^*(H))$ onto the unit ball of $\mathbb{M}_n(VN(G))$, so $S$ is completely isometric.
\end{proof}

\begin{rem} The fact that tensorization by $2 \times 2$ matrices is sufficient to obtain a completely contractive or completely isometric conclusion can be compared with the main result of \cite{JRS}, where the authors prove that a $2$-isometry between noncommutative $L_p$-spaces is necessarily a complete isometry.
\end{rem}

\begin{rem} In the case 2.\ of Theorem \ref{main1}, when $G,H$ are abelian,  we actually obtain that if there is an algebra isomorphism $T:L_1(\hat G) \to L_1(\hat H)$ with  $\| T \|  < \sqrt {{3}/{2}}$, then $G$ and $H$ are topologically isomorphic. (In the abelian case, we do not need tensorization by $2 \times 2$ matrices, because $L_1$ carries here its maximal operator space structure, see \cite{ER1}). Therefore we recover Kalton\&Wood's result with a worse bound, they prove that $\sqrt{2}$ is the optimal bound for $L_1$ over locally compact abelian groups (see Example 1 \cite{KW}). It would be interesting to investigate if $\sqrt { {3}/{2}}$ is the optimal constant $C>1$ in order to have the following implication: for any locally compact groups $G$, $H$, if there is an algebra isomorphism $T:A(G) \to A(H)$ with  $\| T \|_{cb}  < C$, then $G$ and $H$ are topologically isomorphic. Similar investigation could be done for cases (1) and (3).
\end{rem}
\begin{rem} In Example 2 \cite{KW}, N. Kalton and G. Wood showed that $(1+\sqrt {3})/2$ is the optimal constant $C>1$ in order to have the following implication: for any locally compact abelian groups $G$, $H$, for any algebra isomorphism $T:L_1(G) \to L_1(H)$,  if  $\| T \|  < C$, then $T$ is isometric. For this implication too, determining the optimal bound in our non-abelian cases would be interesting as well.
\end{rem}
If we restrict to the class of amenable groups, the bound $\sqrt{5}/2$ in case (1) can be improved:
\begin{prop} Let $G,H$ be two locally compact  groups with $G$ amenable and $T:A(G) \to B(H)$ a nonzero algebra homomorphism. If $\|  T \|_{cb}  < {2}/{\sqrt 3}$, then $\|  T \|_{cb}  \le 1$.
\end{prop}
\begin{proof} Following the proof of Theorem 3.7 in \cite{IS} (and its notation), we obtain a nonzero idempotent $1_{\Gamma_\alpha} \in B(H_d \times G_d)$ with $\|  1_{\Gamma_\alpha} \|  \leq  \|  T \|_{cb} < {2}/{\sqrt 3}$, where $\alpha$ is the map from $H$ to $G$ (induced by $T^*$) and $\Gamma_\alpha$ its graph. But the Fourier-Stieltjes algebra embeds contractively into the set of all completely bounded multipliers of the corresponding Fourier algebras, we can then apply Theorem 3.3 in \cite{S} (which is based on \cite{Li}) to obtain that $\|  1_{\Gamma_\alpha} \|=1$. Hence $\alpha$ is an affine map and so $\|  T \|_{cb}  \leq 1$ by Proposition 3.1 in \cite{IS}.
\end{proof}
In \cite{L}, A.T.-M. Lau considers the second conjugate $A(G)^{**}=VN(G)^*$ equipped with the Arens product. Next corollary can be compared with Theorem 5.3  \cite{L}. But this theorem will be improved in the next section.
\begin{cor} \label{l}Let $G,H$ be two discrete groups. Then $G$ and $H$ are isomorphic if and only if  there exists a surjective algebra isomorphism $T:A(G)^{**} \to A(H)^{**}$ such that $\| \id_{\mathbb{M}_2} \otimes T \| <\sqrt {{3}/{2}}$.
\end{cor}
\begin{proof}  The spectrum of $A(G)^{**}$ possesses the same properties as the spectrum of a Fourier-Stieltjes algebra.
By Lemma 5.1 in \cite{L}, we have the identification $$H=\sigma(A(H)^{**}) \cap VN(H)^{**}_u=\sigma(A(H)^{**}) \cap VN(H)^{**}_r,$$
where $VN(H)^{**}_u$ (resp. $VN(H)^{**}_r$) denotes the unitary group (resp. the invertible group) of the von Neumann algebra $VN(H)^{**}$. Also analyzing the proof of Lemma 5.1 in \cite{L} (the ideal $M$ defined in the proof of this lemma is a regular maximal left ideal of $VN(G)^{*}$ even if $u$ is not invertible), we get $$\sigma(A(G)^{**}) \backslash G \subset A(G)^0,$$
where $A(G)^0$ is the annilihator of $A(G)$ inside $VN(G)^{**}$.
Consider the linear isomorphism $T^*:VN(H)^{**} \to VN(G)^{**}$. Then as in case 3. of the previous theorem, either $T^*(h) \in G$ or $T^*(h) \in A(G)^0$. Still using injectivity of $T$, we get $T^*(\Omega_T)=G$ (because $G$ is discrete). Similar reasoning on $T^{-1*}$ leads to $\Omega_T=H$ and $T^*(H)=G$. Therefore as in case 2., we can define $S$ and apply Lemma \ref{unitmult} to prove that the restriction of $S^*$ to $H$ is a group isomorphism onto $G$.
\end{proof}
\begin{rem} Discreteness of the groups is necessary to use Lemma 5.1 of \cite{L}, see Theorem 4.8 \cite{L}.
\end{rem}



\section{The nearly isometric case}
Recall that the Jordan product of two elements $a,a'$ in a $C^*$-algebra $A$ is defined by:
$$a \circ a'=\frac{aa'+a'a}{2}.$$
It is a well-known result of R.~Kadison \cite{K} that a unital surjective isometry between $C^*$-algebras preserves the Jordan product, the next lemma is an ``approximative'' version of this result.\\
For $C^*$-algebras $A, B$ and a linear map $T:A\to B$,  we define the bilinear map $T^{J}:A^2\to B$ by $$T^{J}(a,a')=T(aa')+T(a'a)-T(a)T(a')-T(a')T(a),$$ for $a,a' \in A$. The next lemma says that a unital nearly isometric bijection between $C^*$-algebras almost preserves the Jordan product.\\
In the next proof, we will need ultraproducts in the category of $C^*$-algebras, for details see for instance \cite{BO},
\cite{BLM}, \cite{P} or \cite{ge-hadwin}.

\begin{lem}\label{jc} For any $\eta>0$, there exists $\rho >0$ such that for any unital  $C^*$-algebras $A, B$, for any unital isomorphism $T:A \to B$,
$\| T \| \leq 1+\rho$ and $\| T^{-1} \| \leq 1+\rho$ imply $\| T^{J} \| < \eta$.
\end{lem}
\begin{proof} Suppose the assertion is false. Then there exists $\eta_0>0$ such that
for every positive integer $n \in \mathbb{N}\backslash \{0\}$, there
is a unital linear map $T_n:A_n \to B_n$ between some unital
$C^*$-algebras satisfying $$\| T_n \| \leq 1+\frac{1}{n},~\| T^{-1}_n \| \leq 1+\frac{1}{n} ~~\hbox{and}~~ \| T_n ^J \| \geq \eta_0.$$ Let $\U$ be a nontrivial ultrafilter on
$\mathbb{N}$, let us denote $A_\U$ (resp. $B_\U$) the $C^*$-algebraic ultraproduct
$\Pi_n  A_n /\U$ (resp. $\Pi_n
 B_n /\U$). Now consider $T_\U: A_\U \to B_\U$ the
ultraproduct map obtained from the $T_n$'s. Hence $T_\U$ is a unital linear surjective isometry between
$C^*$-algebras, so $(T_\U)^J=0$ by the result of Kadison \cite{K} (recalled above), any unital surjective isometry between $C^*$-algebras preserves the Jordan product. This contradicts the hypothesis: $ \| T_n ^J \| \geq \eta_0,$ for every $n$. More precisely, this hypothesis means that there are $u_n, v_n$
in the closed unit ball of $A_n$ such
that
$ \| T_n ^J (u_n, v_n) \| \geq \eta_0$
which implies that $$\big\| (T_\U)^J ([u], [v]) \big\| \geq \eta_0$$
(where $[x]$ denotes the equivalence class of $(x_{n})_n$ in $A_\U$).
\end{proof}

We need to record two ``gap estimates".

\begin{lem}\label{norm-in-vng}
Let $G$ be a locally compact group, denote $\lambda$ its left regular representation and $\omega$ its universal representation. Let
$g_j\in G$ be distinct and $c_j\in \mathbb C$ for $1\le j\le n$.
Then
$$
\| \sum c_j \omega_{g_j} \|_{W^*(G)} \ge \| \sum c_j \lambda_{g_j} \|_{VN(G)} \ge \big(\sum |c_j|^2 \big)^{1/2}.
$$
In particular, for any $g_1,g_2,g_3,g_4 \in G$, if ($g_1=g_3$ and $g_2=g_4$) or ($g_1=g_4$ and $g_2=g_3$), then  $\| \lambda_{g_1}+\lambda_{g_2}- \lambda_{g_3}-\lambda_{g_4}\|_{VN(G)}=0$, otherwise $ \| \lambda_{g_1}+\lambda_{g_2}- \lambda_{g_3}-\lambda_{g_4}\|_{VN(G)} \geq \sqrt 2 $.
\end{lem}
\begin{proof}
Since the universal representation is the direct sum of all unitary representations, the first inequality is obvious. Now let $U\subset G$ be an open neighbourhood of identity such that $g_jU$, $j=1,\dots,n$, are pairwise disjoint.
Set $f = 1_U$ the characteristic function of $U$. Then $\lambda_{g_j}(f) = 1_{g_jU}$, and
\begin{align*}
\| \sum c_j \lambda_{g_j} \|_{VN(G)} &\ge \| \sum c_j \lambda_{g_j} (f)\|_2/\|f\|_2 \\
 &= \Big(\sum \int_{g_jU} |c_j|^2\Big)^{1/2}\Big(\int_U 1\Big)^{-1/2} = \big(\sum |c_j|^2\big)^{1/2}.
\end{align*}
It follows that if $g_1,g_2,g_3,g_4$ do not satisfy the equalities in the statement, then $\| \lambda_{g_1}+\lambda_{g_2}- \lambda_{g_3}-\lambda_{g_4}\|_{VN(G)}$ is estimated from below, depending on the number of equalities, by $\sqrt2$, $2$, $2\sqrt2$ or $\sqrt6$, being always not less than $\sqrt2$.
\end{proof}

\begin{lem}\label{G-separated-in-spectre-of-BG}
Let $g_1,\dots,g_n\in G$ be distinct, and $c_1,\dots,c_n \in \mathbb{Z} \backslash \{0\}$. Let $s_1,\dots,s_m\in \sigma (B(G))$.
If for all $i=1,\dots,m$, $s_i \notin G$, then
$$
1 \leq \| \sum_{k=1}^n c_k \omega_{g_k} - \sum_{i=1}^m s_i \|_{W^*(G)}.
$$
\end{lem}
\begin{proof} As the $\lambda_g$'s are linearly independent, $\sum_{k=1}^n c_k \lambda_{g_k} \neq 0$.
Consider the canonical epimorphism $\pi:W^*(G)\to VN(G)$ satisfying $\pi \circ \omega=\lambda$, then
\begin{align*}
\| \sum_{k=1}^n c_k \lambda_{g_k} \|_{VN(G)} =& \Big\| \,\pi\Big( \sum_{k=1}^n c_k \omega_{g_k} - \sum_{i=1}^m  s_i \Big) \Big\|_{VN(G)} \\
\le & \| \sum_{k=1}^n c_k \omega_{g_k} - \sum_{i=1}^m  s_i \|_{W^*(G)}.
\end{align*}
By the previous lemma, the left-hand side is greater or equal to $1$.
\end{proof}

We are in position to prove Theorem \ref{main2}.

\begin{proof}[Proof of Theorem \ref{main2}]
Let $\varepsilon_0 >0$ be the constant $\rho>0$ corresponding to $\eta=1$ in Lemma \ref{jc}.\\
In case (1), we identify an element of a group with its image under the left regular representation and we can replace $T$ by a unital algebra isomorphism $S$ as in the proof 2.\ of Theorem \ref{main1} with $\|  S \| \|  S^{-1} \| \leq \|  T \| \|  T^{-1} \| <1+\varepsilon_0$.
But since $\|S\|=\|S^*\|\geq\|S^*(e_H)\|=1$, we have in fact that $\|S\|<1+\varepsilon_0$ and $\|S^{-1}\|<1+\varepsilon_0$. Moreover the restriction of $S^*$ maps homeomorphically $H$ onto $G$. We are going to prove that $S^*: VN(H) \to VN(G)$ is actually a $*$-preserving Jordan isomorphism.
For every $h\in H$, by Lemma \ref{jc}, we have:
\begin{align*}
\| S^{*J}(h,h^{-1}) \|=&\| S^*(h h^{-1}) + S^*(h^{-1} h)-S^*(h)S^*(h^{-1})-S^*(h^{-1})S^*(h) \|
\\ =\,& \| 2\, \lambda_{\,1_G}-S^*(h)S^*(h^{-1})-S^*(h^{-1})S^*(h) \| < \eta.
\end{align*}
Then, by Lemma \ref{norm-in-vng}, we get that
$S^*(h)S^*(h^{-1}) = S^*(h^{-1})S^*(h) = \lambda_{\,1_G}$, thus $S^*(h)^{*}=S^*(h^{*})$. As $VN(H)$ is the $w^*$-closure of the linear span of $H$ (and the map $x \in M \mapsto x^* \in M$ is $w^*$-continuous on a general von Neumann algebra $M$) and $S^*$ is obviously $w^*$-continuous, we obtain that $S^*: VN(H) \to VN(G)$ is $*$-preserving, i.e. for any $x \in VN(H)$,  $S^*(x^*)=S^*(x)^*$.
Now let us prove that $S^*$ preserves the Jordan product. By Lemma \ref{jc} again, for all $h_1, h_2\in H$
$$
\| S^{*J}(h_1,h_2) \|=\| S^*(h_1h_2)+ S^*(h_2h_1)-S^*(h_1)S^*(h_2)-S^*(h_2)S^*(h_1) \|< \eta.
$$
By Lemma \ref{norm-in-vng}, it follows that $$
S^*(h_1h_2)=S^*(h_1)S^*(h_2) ~~\hbox{and}~~ S^*(h_2h_1)=S^*(h_2)S^*(h_1)$$ or $$S^*(h_1h_2)=S^*(h_2)S^*(h_1) ~~\hbox{and}~~ S^*(h_2h_1)=S^*(h_1)S^*(h_2).
$$
Consequently, in both cases, we have (with notation of the Jordan product before Lemma \ref{jc}), $$S^*(h_1 \circ h_2)=S^*(h_2) \circ S^*(h_1).$$
By $w^*$-density (since the multiplication is separately $w^*$-continuous on a von Neumann algebra) as above, we obtain that $S^*: VN(H) \to VN(G)$ preserves the Jordan product. We conclude that $S^*: VN(H) \to VN(G)$ is actually a $*$-Jordan isomorphism. From here, we can follow the proof of \cite{Wa} Theorem 2 to obtain that the restriction of $S^*$ to $H$ is topological group isomorphism or anti-isomorphism from $H$ onto $G$.\\
In case (2), we identify an element of a group with its image under the universal representation and we can still replace $T$ by a unital algebra isomorphism $S$ as 3.\ of Theorem \ref{main1} with $\|S\|<1+\varepsilon_0$ and $\|S^{-1}\|<1+\varepsilon_0$. Clearly $S^*: W^*(H) \to W^*(G)$ maps homeomorphically $\sigma(B(H))$ onto $\sigma(B(G))$, but we have to prove that $S^*$ maps homeomorphically $H$ onto $G$ (here we can not follow 3.\ of Theorem \ref{main1} here). For every $h\in H$, by Lemma \ref{jc}:
\begin{align*}
\| S^{*J}(h,h^{-1}) \|=&\| S^*(h h^{-1}) + S^*(h^{-1} h)-S^*(h)S^*(h^{-1})-S^*(h^{-1})S^*(h) \|
\\ =\,& \| 2\, \omega_{\,1_G}-S^*(h)S^*(h^{-1})-S^*(h^{-1})S^*(h) \| < \eta.
\end{align*}
Applying Lemma \ref{G-separated-in-spectre-of-BG}, we get that $S^*(h)S^*(h^{-1}) \in G$ and $S^*(h^{-1})S^*(h) \in G$. Then, by Lemma \ref{norm-in-vng}, we conclude that
$S^*(h)S^*(h^{-1}) = S^*(h^{-1})S^*(h) = \omega_{\,1_G}$. Thus, $S^*(h)$ is invertible and $S^*(h)^{-1}=S^*(h^{-1})$.
By \cite{Wa} Theorem 1, we have the identifications $$H=\sigma(B(H)) \cap W^*(H)_u=\sigma(B(H)) \cap W^*(H)_r,$$
where $W^*(H)_u$ (resp. $W^*(H)_r$) denotes the unitary group (resp. the invertible group) of the von Neumann algebra $W^*(H)$, hence $S^*(h)$ is actually a unitary, $S^*(h)^{*}=S^*(h^{*})$ and $S^*(H)=G$.
The rest of the proof follows case $(1)$ and finish with the proof of \cite{Wa} Theorem 3.
\end{proof}

In \cite{L}, A.T.-M. Lau considers the second conjugate $A(G)^{**}=VN(G)^*$ equipped with the Arens product. He proved (see Theorem 5.3  \cite{L}) that $G$ and $H$ are isomorphic if there is an order preserving isometric algebra isomorphism $T:A(G)^{**} \to A(H)^{**}$. The next corollary improves his theorem, even in the isometric case. For the proof, one just needs to combine the beginning of the proof of Corollary \ref{l} and the proof of Theorem \ref{main2} (with computation now inside of the von Neumann algebra $VN(G)^{**}$), we leave it to the reader.

\begin{cor}\label{lau} Let $G,H$ be two discrete groups. There exists a universal constant $\varepsilon_0>0$ such that: if there is an algebra isomorphism $T:A(G)^{**} \to A(H)^{**}$ with  $\|  T \| \,\|T^{-1}\|  <1+\varepsilon_0$, then $G$ and $H$ are isomorphic.
\end{cor}

\begin{rem}  Our proof of Theorem \ref{main2} does not yield an explicit value of $\varepsilon_0$, but we can estimate it from above. By Example 1 of \cite{KW}, there is an isomorphism $T:A(\mathbb{Z}_4)\to A(\mathbb{Z}_2\times \mathbb{Z}_2)$ of norm $\sqrt2$ and one can calculate that the distortion $\|  T \| \ \|T^{-1}\| =2$, hence $\varepsilon_0\le1$.
\end{rem}

\begin{rem} To our knowledge, Kalton and Wood's algebra isomorphism mentioned in the previous remark is the only known computation of the distortion of an algebra isomorphism between Fourier algebras. As it involves only abelian groups, we find it interesting to give another example involving a non-abelian group. We claim that
there exists an algebra isomorphism between $A(\mathbb{Z}_6)$ and $A(S_3)$ of distortion $2$ (where $S_3$ denotes the symmetric group on three elements).\\
A function $f\in A(\mathbb{Z}_6)$ is represented by its values: $(f_0,\dots,f_5)$. The group $\mathbb{Z}_6$ has 6 characters: $\chi_j(k)=e^{i\pi jk/3}$, $j=0,\dots,5$.
Hence for $j=0,1,\dots,5$, its Fourier coefficients are $\hat f_j = \sum_{k=0}^5 f_k\, e^{i\pi k j/3}$ and
$$\|f\|_{A(\mathbb{Z}_6)} = \sum_{j=0}^5 |\hat f_j|.$$
The group $S_3$ is generated by the transposition $s=(12)$ and the cycle $r=(123)$. We have thus $S_3=\{\id,s,r,sr,r^2,sr^2\}$.
Define $\phi:S_3 \to \mathbb{Z}_6$ by $\phi(\id)=0$, $\phi(s)=1$, $\phi(r)=2$, $\phi(sr)=3$, $\phi(r^2)=4$, $\phi(sr^2)=5$. Now define a surjective algebra isomorphism $\Phi: A(\mathbb{Z}_6)\to A(S_3)$ by $\Phi(f)(h)=f(\phi(h))$, $f \in A(\mathbb{Z}_6)$, $h \in S_3$. But from now on, for simplicity, we identify $k \in \mathbb{Z}_6$ with $\phi^{-1}(k) \in S_3$.
The group $S_3$ has three irreducible representations: $\chi_0$, $\chi_3$ and $\pi$ (of multiplicity $2$) defined by
$$
\pi(s)= \left[ \begin{array}{cc} 0 & 1 \\ 1 &  0 \end{array}\right]
\quad \text{and} \quad \pi(r)= \left[ \begin{array}{cc} e^{i \frac{2\pi}{3}} & 0 \\  0 &
    e^{-i \frac{2\pi}{3}} \end{array}\right].
$$
The coefficients of $\pi$, as functions on $S_3=\{\id,s,r,sr,r^2,sr^2\}$, are:
$$
\begin{array}{cccccccc}
\pi_{11}& = & (\,1 & 0 & e^{2\pi i/3}& 0 & e^{-2\pi i/3}& 0\,)\\
\pi_{21}& = & (\,0 & 1 & 0 & e^{2\pi i/3}& 0 &e^{-2\pi i/3}\,)\\
\pi_{12}& = & (\,0 & 1 & 0 & e^{-2\pi i/3}& 0 &e^{2\pi i/3}\,)\\
\pi_{11}& = & (\,1 & 0 & e^{-2\pi i/3}& 0 & e^{2\pi i/3}& 0\,)\\
\end{array}
$$
One verifies that
\begin{align*}
\pi_{11} &= (\chi_1+\chi_4)/2;\\
\pi_{21} &= (\chi_1-\chi_4)e^{-\pi i/3}/2;\\
\pi_{12} &= (\chi_2-\chi_5)e^{-2\pi i/3}/2;\\
\pi_{22} &= (\chi_2+\chi_5)/2.
\end{align*}
It follows that for any $f \in A(S_3)$,
$$
\pi(f) = \frac12 \left[ \begin{array}{cc} \hat f_1+\hat f_4 & e^{-2\pi i/3} (\hat f_2 - \hat f_5) \\
                                          e^{-\pi i/3} (\hat f_1-\hat f_4) & \hat f_2+\hat f_5 \end{array} \right].
$$
Since the left regular representation $\lambda_{S_3}$  is unitarily equivalent to $\chi_0\oplus\chi_3\oplus \pi \oplus \pi$,
$$
\|f\|_{A(S_3)} = |\chi_0(f)| + |\chi_3(f)| + 2\, \|\pi(f)\|_1 = |\hat f_0|+|\hat f_3|+2\,\|\pi(f)\|_1,
$$
where for $1 \le p < \infty$, $\|x\|_p = (\Tr((x^*x)^{p/2}))^{1/p}$ is the $p$-Schatten norm of $x \in \mathbb{M}_2$. Clearly $\|x\|_2=(|x_{11}|^2+|x_{12}|^2+|x_{21}|^2+|x_{22}|^2)^{1/2}$.
Also let us recall the \footnote{Denoting $\lambda_{1}$,$\lambda_{2}$ the eigenvalues of $x^*x$, then
$\lambda_1+\lambda_2=\|x\|_2^2$
and
 $\lambda_1\lambda_2 = \det (x^*x)=\vert \det (x) \vert^2$. Moreover $\|x\|_1 =\Tr((x^*x)^{1/2})= \sqrt{\lambda_1}+\sqrt{\lambda_2}
 = \big( \lambda_1 + \lambda_2 + 2\sqrt{\lambda_1\lambda_2} \big)^{1/2}$. } identity: for $x  \in \mathbb{M}_2$, $$\|x\|_1=\big( \|x\|_2^2 + 2 \vert \det (x) \vert \big)^{1/2}.$$
 In particular, $\|x\|_2 \le \|x\|_1$, hence, we get \begin{align*}
\|\pi(f)\|_1 &\ge \big( |\hat f_1+\hat f_4|^2 + |\hat f_1-\hat f_4|^2 + |\hat f_2+\hat f_5|^2 + |\hat f_2-\hat f_5|^2 \big) ^{1/2}/2\\
&= \big( |\hat f_1|^2+|\hat f_4|^2 + |\hat f_2|^2+|\hat f_5|^2 \big) ^{1/2} / \sqrt2
 \\& \ge (|\hat f_1|+|\hat f_4| + |\hat f_2|+|\hat f_5|)/(2\sqrt2).
\end{align*}
This implies that $\|f\|_{A(\mathbb{Z}_6)}/\sqrt2 \le \|f\|_{A(S_3)}$.
 For the other estimate, we apply the aforementioned identity to $x=2\pi(f)$. Note first that
$$
\|2\pi(f)\|^2_2 = 2(|\hat f_1|^2+|\hat f_4|^2 + |\hat f_2|^2+|\hat f_5|^2)
$$ and
\begin{align*}
\vert \det (2\pi(f)) \vert
 & = |(\hat f_1+\hat f_4)(\hat f_2+\hat f_5) + ({\hat f}_2 - {\hat f}_5)({\hat f}_1-{\hat f}_4)|
 \\& = |\hat f_1\hat f_2+\hat f_1\hat f_5+\hat f_2\hat f_4+\hat f_4\hat f_5 + \hat f_1\hat f_2 - \hat f_1\hat f_5
 - \hat f_2\hat f_4 + \hat f_4\hat f_5|
 \\& = 2|\hat f_1\hat f_2+\hat f_4\hat f_5|.
\end{align*}
This gives
\begin{align*}
\|2\pi(f)\|_1 &
 = \Big( 2(|\hat f_1|^2+|\hat f_4|^2 + |\hat f_2|^2+|\hat f_5|^2)
  + 4|\hat f_1\hat f_2+\hat f_4\hat f_5| \Big)^{1/2}
\\& \le
\sqrt2 \Big( |\hat f_1|^2+|\hat f_4|^2 + |\hat f_2|^2+|\hat f_5|^2
  + 2 (|\hat f_1\hat f_2|+|\hat f_4\hat f_5|) \Big)^{1/2}
\\&= \sqrt2 \Big( (|\hat f_1|+|\hat f_2|)^2+(|\hat f_4|+|\hat f_5|)^2\Big)^{1/2}
\\& \le \sqrt2 ( |\hat f_1| +|\hat f_2| + |\hat f_4| + |\hat f_5| ),
\end{align*}
which implies $\|f\|_{A(S_3)} \le \sqrt2 \|f\|_{A(\mathbb{Z}_6)}$.\\
Finally,  we have proved
$$
\|f\|_{A(\mathbb{Z}_6)}/\sqrt2 \le \|f\|_{A(S_3)} \le \sqrt2 \|f\|_{A(\mathbb{Z}_6)}.
$$
Moreover, both equalities are attained. For the first one, set $\hat f_k=1$ for $k=1,2,4$, $\hat f_5=-1$ and $\hat f_0=\hat f_3=0$, then $\|f\|_{A(\mathbb{Z}_6)}=4$ and \footnote{$\|2\pi(f)\|_1
 = \sqrt2\Big( |\hat f_1|^2+|\hat f_4|^2 + |\hat f_2|^2+|\hat f_5|^2
  + 2|\hat f_1\hat f_2+\hat f_4\hat f_5| \Big)^{1/2}
  =2\sqrt2,$}
 $\|f\|_{A(S_3)} = \|2\pi(f)\|_1 =2\sqrt2$.
For the second equality, take $\hat f_k=0$ for $k\ne1$ and $\hat f_1=1$, then
$\|f\|_{A(\mathbb{Z}_6)} = 1$ and \footnote{$\pi(f) = \frac12\left[\begin{array}{cc}1&0\\e^{-i\pi/3}&0\end{array}\right]$, so $\sqrt{\pi(f)^*\pi(f)} = \frac12 \left[\begin{array}{cc}\sqrt2&0\\0&0\end{array}\right]$.}  $\|f\|_{A(S_3)} = \|2\pi(f)\|_1 = \sqrt2$. We conclude that $\|\Phi\|=\|\Phi^{-1}\|=\sqrt2$.


\end{rem}

\end{document}